\documentclass[11pt,a4paper]{article}
\usepackage{amsmath,amssymb,amsthm,amsfonts,latexsym,graphicx,subfigure}
\usepackage[all,import]{xy}
\usepackage[numbers,sort&compress]{natbib}
\usepackage{indentfirst}
\usepackage{fancyhdr,enumerate}
\usepackage{bbm,color}
\usepackage{tikz}%mathpazo
\usetikzlibrary{shapes.geometric, arrows}
\usepackage{accents,cases}
\usepackage{multirow}%\usepackage{eps2pdf}
\usepackage[lined, ruled, linesnumbered, longend]{algorithm2e}

\usepackage{array,float}
\newcommand{\PreserveBackslash}[1]{\let\temp=\\#1\let\\=\temp}
\newcolumntype{C}[1]{>{\PreserveBackslash\centering}p{#1}}
\newcolumntype{R}[1]{>{\PreserveBackslash\raggedleft}p{#1}}
\newcolumntype{L}[1]{>{\PreserveBackslash\raggedright}p{#1}}

\makeatletter
\def\wbar{\accentset{{\cc@style\underline{\mskip8mu}}}}

\makeatother

\newcommand{\R}{\ensuremath{\mathbb{R}}}
\theoremstyle{plain}
\newtheorem{theorem}{Theorem}

\newtheorem{defn}{Definition}[section]

\newtheorem{lemma}{Lemma}%[section]
\newtheorem{remark}{Remark}[section]

\newtheorem{cor}{Corollary}[section]
\newtheorem{pro}{Proposition}[section]
\newtheorem{example}{Example}[section]

\allowdisplaybreaks

\begin{document}

\title{Spectral bounds of multi-way Cheeger constants via cyclomatic number}
\author{Chuanyuan Ge}

\footnotetext{School of Mathematical Sciences, 
University of Science and Technology of China, Hefei 230026, China. \\
Email address:
{\tt gechuanyuan@mail.ustc.edu.cn} }
%\footnotetext[2]{School of Mathematical Sciences, 
%University of Science and Technology of China, Hefei 230026, China. \\
%Email address:
%{\tt spliu@ustc.edu.cn}
%}
\date{}\maketitle
\begin{abstract}
As a non-trivial extension of the celebrated Cheeger inequality, the higher-order Cheeger inequalities for graphs due to Lee, Oveis Gharan and Trevisan provide for each $k$ an upper bound for the $k$-way Cheeger constant in forms of $C(k)\sqrt{\lambda_k(G)}$, where $\lambda_k(G)$ is the $k$-th eigenvalue of the graph Laplacian and $C(k)$ is a constant depending only on $k$. In this article, we prove some new bounds for multi-way Cheeger constants. By shifting the index of the eigenvalue via cyclomatic number, we establish upper bound estimates with an absolute constant instead of $C(k)$. This, in particular, gives a more direct proof of Miclo's higher order Cheeger inequalities on trees. We also show a new lower bound of the multi-way Cheeger constants in terms of the spectral radius of the graph. The proofs involve the concept of discrete nodal domains and a probability argument showing generic properties of eigenfunctions.
\end{abstract}

\section{Introduction}
Given any finite graph $G=(V,E)$, let us list the eigenvalues of its Laplacian by counting multiplicity as follows:
\[\lambda_1(G)\leq \lambda_2(G)\leq \cdots \leq \lambda_{|V|}(G).\] For simplicity, in this introduction section, we restrict ourselves to normalized graph Laplacian where its vertex measure is equal to its degree. 

A fundamental result in spectral graph theory tells that $\lambda_{2}(G)=0$ if and only if the graph $G$ is disconnected. In general, the Cheeger constant $\rho_2(G)$ measures the connectivity of graphs quantitatively. 
%When $G$ is connected, the Cheeger constant $\rho_2(G)$ is introduced to measure the connectivity of graphs \cite{Fiedler73}, and later it can be used to measure the quality of clustering \cite{KVSA04}. 
The Cheeger inequality \cite{Alon86, AM85, Dodziuk, Fiedler73} relates the Cheeger constant $\rho_2(G)$ and the eigenvalue $\lambda_2(G)$ as below
$$ \frac{\lambda_2(G)}{2}\leq \rho_2(G)\leq \sqrt{2\lambda_2(G)}.$$
We remark that the proof of the first inequality above is relatively straightforward while the proof of the second inequality is far from obvious. The proof actually provides a spectral partitioning algorithm. Cheeger inequality plays a crucial role in various research areas, such as explicit constructions of expander graphs \cite{Alon86,Lubotzky,PV11,Tanner}, graph coloring \cite{AG84,AK97}, image segmentation \cite{SM00,TM06} and web search \cite{BP98,Kleinberg99}.

Concerning higher order eigenvalues, we have $\lambda_{k}(G)=0$ if and only if the graph $G$ has at least $k$ connected components. The multi-way Cheeger constants $\rho_k(G), k=1,2,\ldots,|V|$ (for precise definition see Section \ref{sec:pre}), introduced by Miclo \cite{Miclo08}, measure how difficult it is to partition the graph into $k$-clusterings. It was conjectured that there hold related higher-order Cheeger inequalities for weighted graphs \cite{DJM12}. In 2012, Lee, Oveis Gharan and Trevisan used the method of random metric partitions to confirm this conjecture \cite{LGT14}. They proved that
\begin{equation}\label{eq:higer}
   \frac{\lambda_k(G)}{2}\leq \rho_k(G)\leq Ck^2\sqrt{\lambda_k(G)},
\end{equation}
where $C$ is a universal constant which does not depend on $G$. This result gives theoretical support for spectral clustering algorithms, see e.g. \cite{NJW,SFNWNBB,Luxburg}. 
Miclo \cite{Miclo15} generalized higher-order Cheeger inequalities on graphs to weighted Riemannian manifolds and used them as a crucial tool to solve the conjecture of Simon and H\o egh-Krohn \cite{SH72}. The higher-order Cheeger inequalities have further been extended to signed graphs \cite{Liu15,AL20} and discrete and continuous magnetic Laplacians \cite{LLPP, Yan22}.

As in the case of $k=2$, the main difficulty of proving  (\ref{eq:higer}) lies in the second inequality. In general, the dependence on $k$ of the constant in the right-hand side of (\ref{eq:higer}) is necessary (see \cite[Section 4.4]{LGT14}). If we allow the index of the eigenvalues to be larger than the index of multi-way Cheeger constants, then the term $k^2$ can be improved to $\log k$. Exactly, it is proved \cite{LGT14,LRTV} that there exists absolute constants $c<1$ and $C_0$ such that 
\begin{equation}\label{eq:logk}
    \rho_{ck}(G)\leq \sqrt{C_0\log k\lambda_{k}(G)}.
\end{equation}
From the perspective of giving an upper bound on higher-order constants by eigenvalues, for large $k$, the inequalities \eqref{eq:higer} and \eqref{eq:logk} can be trivial because of the term $\log k$ or $k^2$, see Example \ref{ex:1}.

When the graph is a tree or a cycle, the dependence on $k$ of the constant is not necessary. In fact, it was proved \cite{DJM12,Miclo08} that $\rho_k(G)\leq C\sqrt{\lambda_k(G)}$ holds for a tree or a cycle with an absolute constant $C$. In this article, we first prove it holds for any graph that  $$\rho_{k-l}(G)\leq \sqrt{2\lambda_k(G)},$$ where $\ell=|E|-|V|+1$ is the cyclomatic number of $G$ \cite{Diestel}. For a more precise expression, the reader can see Theorem \ref{thm:main}. Secondly, we give an upper bound of $\rho_k(G)$ by the eigenvalue $\lambda_k(G)$ for some special product graph and special eigenvalues with an absolute constant independent to $k$, see Theorem \ref{thm:main3}. Our proof of these two theorems uses the concept of discrete nodal domains and a crucial argument on generic properties of eigenfunctions (see Lemma \ref{lemma:nozero}). The lemma shows that for any graph, we give a continuous perturbation on the edge weight and potential function, we have that all eigenvalues of its normalized Laplacian are simple with probability $1$. Moreover, when the graph is connected, all eigenfunctions of its normalized Laplacian have no zeros with probability $1$. The theory that gives a perturbation to make the spectrum of an operator to be simple is also important in other models. A remarkable and well-known result proved by Uhlenbeck shows that the spectrum of the Laplacian on $C^k$-Riemannian is always simple if we give a small perturbation on its metric \cite{Uhlenbeck}. This type of question is also considered on graphs and quantum graphs \cite{BL17,BL18,Friedlander,PPP}. Additionally, our results imply Miclo's result on trees, see Corollary \ref{cor:main}. Besides the upper bounds, we give a new lower bound of $\rho_k$ that 
$$\rho_k(G)\geq \left(1-\frac{1}{k}\right)\min \{\lambda_2(G), 2-\lambda_n(G)\}.$$
Compared with the left hand of Inequality \eqref{eq:higer}, this bound improved the constant $\frac{1}{2}$ but weakened $\lambda_k(G)$. 

A signed graph is a graph with a signature $\sigma:E\to \{+1,-1\}$.  The spectral theory of signed graphs has recently led to several breakthroughs in theoretical computer science and combinatorial geometry, such as the sensitivity conjecture \cite{Huang} and the spherical two-distance set problem \cite{JTYZZ}. One of the most important issues in the study of signed graphs is the Bilu-Linial conjecture \cite{BL06}. They conjectured that for any $d$-regular graph with $d\geq 3$, there exists a signature such that its spectral radius of the signed adjacency matrix will not exceed $2\sqrt{d-1}$. If this conjecture holds, then we can construct infinite families of Ramanujan graphs of every degree bigger than $2$ by using the $2$-lift. In 2015, Marcus, Spielman and Srivastava proved that this conjecture holds for the largest eigenvalues of the signed adjacency matrix by using the method of interlacing families \cite{MSS15}. Therefore, it should be necessary to consider higher-order Cheeger inequalities on signed graphs since they are closely related to eigenvalues. In Section \ref{sec:signed}, we review the definition of multi-way Cheeger constants on signed graphs and show that our Theorem \ref{thm:main} also holds with the signed multi-way Cheeger constants. 

 This article is structured as follows. In Section \ref{sec:pre}, we give the definition of higher-order Cheeger constants. We present Theorem \ref{thm:main}, Theorem \ref{thm:main3} and Theorem \ref{thm:main4} and their proof in Section \ref{sec:main1}, Section \ref{sec:main3} and Section \ref{sec:main4}, respectively. At the end of Section \ref{sec:main1}, we give an example to illustrate that for some graphs, our Theorem \ref{thm:main} is better than Inequality \eqref{eq:higer}. Finally, we give a remark the Theorem \ref{thm:main} can be extended to the signed graphs in Section \ref{sec:signed}.

\section{Preliminaries}
\label{sec:pre}
In this article, we only consider simple graphs i.e., graphs without self-loops and multi-edges. Additionally, we default that all graphs in this article have no isolated vertices. Let $G=(V,E,w,\mu,\kappa)$ be a finite weighted graph with an edge weight $w:E\to \R_{>0}$, a vertex measure $\mu:V\to\R_{>0}$ and a potential function  $\kappa:V\to\R$. We always assume $V=\{1,2,\ldots,n\}$ to simplify our exposition. We denote $w(\{i,j\})$, $\kappa(i)$ and $\mu(i)$ by $w_{ij}$,  $\kappa_i$ and $\mu_i$ for simplicity.  We sometimes write $i\sim j$ to imply $\{i,j\}\in E$. The weighted degree of a vertex $i$ is defined by $d(i)=\sum_{j\sim i}w_{ij}$. We define $\tau_G:=\max_{i\in V}\frac{d(i)}{\mu_i }$. In some other literature, it is assumed that $\mu_i=d(i)$ for any $i\in V$. In this situation, we have $\tau_G=1$.

 Given a weighted graph $G=(V,E,w,\mu,\kappa)$, its adjacency matrix is an $n\times n$ matrix $A=\{a_{ij}\}_{1\leq i,j\leq n}$ where $a_{ij}=w_{ij}$ if $i\sim j$ and $a_{ij}=0$ if $i\nsim j$.
 The Laplacian of $G$ is defined by $L=D+K-A$ where $D$ and $K$ are both diagonal matrices with $D_{ii}=d(i)$ and $K_{ii}=\kappa_i$. The normalized Laplacian of $G$ is defined by  $\mathfrak{L}=M^{-1}L$ where $M$ is a diagonal matrix with $M_{ii}=\mu_{i}$. 
 
In some other literature, it is assumed the normalized Laplacian is $\mathfrak{L}'=M^{-\frac{1}{2}}LM^{-\frac{1}{2}}$. It is worth noting that the eigenvalues of  $\mathfrak{L}$ and $\mathfrak{L}'$ are the same. Exactly, for any eigenfunction $f$ of $\mathfrak{L}$ corresponding to an eigenvalue $\lambda$, $M^{\frac{1}{2}}f$ is an eigenfunction  $\mathfrak{L}'$ corresponding to $\lambda$.

In this article, we always assume the all eigenvalues $\{\lambda_{k}(G)\}_{k=1}^n$ of $\mathfrak{L}$ are sort in ascending order, i.e., 
$$\lambda_1(G)\leq\lambda_2(G)\cdots\leq\lambda_{n-1}(G)\leq \lambda_n(G). $$

Given a subset $A\subset V$, we denote the complement of $A$ by $\overline{A}:=V\setminus A$. The conductance of $A$ is defined by $$\Phi_{G}(A)=\frac{\sum_{i\in A,j\in \overline{A}}w_{ij}}{\sum_{i\in A}\mu_i}.$$ For $1\leq k \leq n$, the $k$-way Cheeger constant $\rho_k(G)$ is defined as follows, 
$$\rho_k(G)=\min_{(A_1,A_2,\ldots,A_k)\in \mathcal{D}_k}\max_{1\leq i\leq k}\Phi_{G}(A_i),$$
where $\mathcal{D}_k:=\{(A_1,A_2,\ldots,A_k):\emptyset\neq A_i\subset V,A_i\cap A_j= \emptyset\}$.

It is direct to calculate similarly as in \cite[Lemma 7]{GLZ23} to get the monotonicity of the multi-way Cheeger constants.
\begin{pro}\label{pro:mono}
    For any integer $1\leq k\leq n-1$, we have $\rho_k(G)\leq \rho_{k+1}(G)$.
\end{pro}
In the proof of the main results, we use the following concepts.
\begin{defn}[Nodal domain]
       Let $G=(V,E)$ be a graph and $f:V\to \mathbb{R}$ be a function on $V$.

\begin{itemize}
  \item [(i)]  A subset $S\subset V$ is a strong nodal domain of $f$ on $G$ if it is a connected component of $G'=(V',E')$ where $V'=\{x\in V:f(x)\neq 0\}$ and $E'=\left\{ \{x,y\}\in E:f(x)f(y)>0 \right\}$. We denote the number of the strong nodal domains of $f$ by $\mathfrak{S}(f)$.
  \item [(ii)]  A subset $W\subset V$ is a weak nodal domain of $f$ on $G$ if it is a connected component of $G''=(V'',E'')$ where $V''=V$ and $E''=\left\{ \{x,y\}\in E:f(x)f(y)\geq 0 \right\}$. We denote the number of the weak nodal domains of $f$ by $\mathfrak{W}(f)$.
\end{itemize}

\end{defn}
 \begin{remark}
     This concept is a discrete version of the continuous situation.  Courant’s nodal domain theorem states that for any $k$-th eigenfunction of a self-adjoint second-order elliptic differential operator on domain $D$, its zeros can not divide the $D$ into more than $k$ different subdomains \cite{Courant23,Courant53}. This type of upper bounds also holds in the graph setting. To be precise, if $f$ is an eigenfunction corresponding to $k$-th eigenvalue $\lambda_k$ of graph Laplacian, we have $\mathfrak{S}(f)\leq k+r-1$ and $\mathfrak{W}(f)\leq k+c-1$, where $r$ is the multiplicity of $\lambda_k$ and $c$ is the number of connected components of the graph \cite{DGLS01,Fiedler75,Fiedler752,Lovasz}.
 \end{remark}
\begin{remark}
    The study of nodal domains in random graphs is also important. In 2010, Dekel, Lee, and Linial first studied the nodal domains in the model of random graphs. They proved that asymptotically almost surely holds for every graph, the two largest weak nodal domains of each eigenfunction nearly cover all vertices of the graph and the two largest strong nodal domains almost cover all vertices of the graph whose are not zeros of this eigenfunction\cite{DLL}. In 2019, Huang and Rudelson proved that the sizes of these two large nodal domains are approximately equal to each other with high probability \cite{HR}. Moreover, the theory of nodal domains in random graphs is closely related to the discrete analog of Berry’s conjecture \cite{GMMS}.
\end{remark}

\section{Higher-order Cheeger inequalities}
\label{sec:main1}
In this section, we prove the following higher-order Cheeger inequalities.
\begin{theorem}\label{thm:main}
    Let $G=(V,E,w,\mu,\kappa)$ be a connected weighted graph with $\kappa\geq 0$. Then we have
     $$\rho_{k-l}(G)\leq \sqrt{2\tau_G\lambda_k(G)},$$ 
 where $\ell=|E|-|V|+1$ is the cyclomatic number of $G$.
\end{theorem}

This theorem implies the result of higher-order Cheeger inequalities on tress which was proved by Miclo \cite{DJM12,Miclo08}.
\begin{cor}\label{cor:main}
    Let $G=(V,E,w,\mu,\kappa)$ be a weighted tree with $\kappa\geq 0$.  Then we have $$\rho_{k}(G)\leq \sqrt{2\tau_G\lambda_k(G)}.$$ 
\end{cor}
 Before proving Theorem \ref{thm:main}, we need some preparation. Given any matrix $M$, we denote $M^{(ij)}$ be the matrix obtained by removing the $i$-th row and $j$-th column from $M$.  
\begin{lemma}\label{lemma:det}
    Given an $n\times n$ matrix $M=\{m_{ij}\}_{1\leq i,j\leq n}$, let $$\tilde{M}:=M+\left(
\begin{array}{cccc}
	\xi_{1} & 0 & \cdots &  0 \\
	0 & \xi_{2} &\cdots &  0  \\
	\vdots & \vdots& \ddots & \vdots\\
	0 &  0 &\cdots & \xi_{n} \\
\end{array}
\right),$$ where $\{\xi_{i}\}_{i=1}^n$ are independent continuous random variables. Then we have $$\mathrm{Pr}[\;\mathrm{det}(\tilde{M})\neq 0\;]=1.$$
\end{lemma}
\begin{proof}
  Denote the order of the matrix $M$ by $n$. We prove this lemma by induction on $n$. 
  
  When $n=1$, we compute $$\text{Pr}[\; \text{det}(\tilde{M})\neq 0 \;]=\text{Pr}[\; \xi_1\neq -m_{11}\; ]=0.$$
    This shows the lemma is true for $n=1$.

    We assume that this lemma holds for any matrix with an order equal to or less than  $n-1$. Next, we prove this lemma holds for any $n\times n$ matrix $M$. 
    By induction, we have
    \begin{equation*}
        \begin{aligned}
            \text{Pr}[\;\text{det}(\tilde{M})\neq 0\;]&=\text{Pr}[\;\text{det}(\tilde{M})\neq 0,\;\text{det}\left(\tilde{M}^{(11)}\right)\neq 0\;] +\text{Pr}[\;\text{det}(\tilde{M})\neq 0,\;\text{det}\left(\tilde{M}^{(11)}\right)= 0\;]\\
            &=\text{Pr}[\;\text{det}(\tilde{M})\neq 0,\;\text{det}\left(\tilde{M}^{(11)}\right)\neq 0\;].
        \end{aligned}
    \end{equation*}
   Conditioning on $\xi_{i}=s_i$ for $i=2,\ldots,n,$ such that $\mathrm{det}(\tilde{M}^{(11)})\neq 0$. We compute
\begin{equation*}
    \begin{aligned}
       \text{Pr}[\;\text{det}(\tilde{M})\neq 0\;]&=\text{Pr}[\;(\xi_1+m_{11})\text{det}(\tilde{M}^{(11)})+C\neq 0\;]\\
       &=\text{Pr}[\;\xi_1\neq \frac{-C}{\text{det}(\tilde{M}^{(11)})}-m_{11}\;\;]=1,
    \end{aligned}
\end{equation*}
where the constant $C=\sum_{j=2}^{n}(-1)^{j+1}m_{1j}\mathrm{det}\left(\tilde{M}^{(1j)}\right)$. 
 This concludes the proof.
\end{proof}
We use the above lemma to prove the main lemma which is used in the proof of the main result in this section. 

\begin{lemma}\label{lemma:nozero}
    Let $G=(V,E)$ be a graph. Let $\tilde{M}=(\xi_{ij})_{1\leq i,j\leq n}$ be a symmetric random matrix which the upper triangular coefficients $\{\xi_{ij}\}_{1\leq i\leq j\leq n}$ are independent. Additionally, for $1\leq i\leq j\leq n $, we assume $\xi_{ij}$ satisfies the follows condition:
    $$\xi_{ij}:=\begin{cases}
0 &\text{if}\;i\nsim j, \\
\text{a continuous random variable} & \text{if}\; i\sim j \text{ or }i=j.
\end{cases}$$ Then the probability of the event that all eigenvalues of $\tilde{M}$ are simple is one. Moreover, if $G$ is connected, then the probability of the event that all eigenfunctions of $\tilde{M}$ have no zeros is one. 
\end{lemma}
\begin{remark}
    This lemma is similar to the main result in \cite{PPP}, but our lemma can admit a permutation on potential functions. This makes our proof more concise.
\end{remark}
\begin{proof}
    We prove this lemma by induction. When $|V|=1$, it is direct to check that this lemma is true. We assume that this lemma is true  for any graph $G_0=(V_0,E_0)$ with $|V_0|\leq n-1$, then we prove this lemma for any graph $G=(V,E)$ with $|V|=n$.% Suppose $P_{ij}$ is the probability measure on $\mathbb{R}$ induced by $\xi_{ij}$.

We assume that $G$ has $k$ connected components which are denoted by $G_{l}=(V_l,E_l)$ for $l=1,\ldots,k$. Denote $|V_l|$ by $n_l$. Without loss of generality, we assume that the random matrix $\tilde{M}$ has the following form $$\tilde{M}=\left(
\begin{array}{cccc}
	\tilde{M}_{1} & 0 & \cdots &  0 \\
	0 & \tilde{M}_{2} &\cdots &  0  \\
	\vdots & \vdots& \ddots & \vdots\\
	0 &  0 &\cdots &  \tilde{M}_{k} \\
\end{array}
\right),$$ 
where $\tilde{M}_l$ is the random matrix corresponding to $G_l$. By definition, $\tilde{M}_l$ is an $n_l \times n_l$ random matrix.

 By direct computing, we get 
 \begin{equation*}
        \begin{aligned}
       \text{Pr}[\text{ the spectrum of $\tilde{M}$ is not simple }]
            &\leq \sum_{l=1}^{k}\text{Pr}[\mathbb{A}_l ]+\sum_{i\neq j}\text{Pr}[ \text{ $\tilde{M}_i$ and $\tilde{M}_j$ have same eigenvalues }]\\
            & =\sum_{i\neq j}\text{Pr}[\text{ $\tilde{M}_i$ and $\tilde{M}_j$ have same eigenvalues }]
        \end{aligned}
    \end{equation*}
where $\mathbb{A}_l$ is the event that the spectrum of $\tilde{M}_l$ is not simple and the above equality is by induction.  

We only prove that $ \text{Pr}[\text{$\tilde{M}_1$ and $\tilde{M}_2$ have same eigenvalues}]=0$, since for any $i\neq j$, the proof of $ \text{Pr}[\text{$\tilde{M}_i$ and $\tilde{M}_j$ have same eigenvalues}]=0$ is the same.
Conditioning on all random variables except $\{\xi_{ii}\}_{i=1}^{n_1}$. Assume $\{\lambda_l\}_{l=1}^{n_2}$ are all eigenvalues of $\tilde{M}_2$ with multiplicity.
We compute
\begin{equation*}
    \begin{aligned}
        \text{Pr}[\text{$\tilde{M}_1$ and $\tilde{M}_2$ have same eigenvalues}]&\leq \sum_{l=1}^{n_2}\text{Pr}[ \lambda_l\text{ is an eigenvalue of }\tilde{M}_1]\\  
        &= \sum_{l=1}^{n_2}\text{Pr}[ \mathrm{det}(\lambda_lI-\tilde{M}_1)= 0 ]\\   
        &=0.
    \end{aligned}
\end{equation*}
The last above equality is by Lemma \ref{lemma:det}. This proves the first statement of this lemma. 

Next, we assume $G$ is connected. We have
\begin{equation*}
        \text{Pr}[\text{ there exists an eigenfunction of $\tilde{M}$ has zeros }]\leq \sum_{i=1}^{n}\text{Pr}[\mathbb{B}_i],\\    
\end{equation*}
where $\mathbb{B}_i$ is the event that there exists an eigenfunction of $\tilde{M}$ which is zero on vertex $i$. 
It is sufficient to prove the second statement of this lemma by proving the right-hand side of the above inequality is zero. We only prove that $\text{Pr}[\mathbb{B}_n]=0$, because it is the same to prove 
$\text{Pr}[\mathbb{B}_i]=0$ for any $1\leq i \leq n-1$.  We assume there are $m$ connected components of the graph obtained by removing the vertex $n$ from $G$ and denote them by $\left\{G'_i=(V_l',E_l')\right\}_{l=1}^m$. Without loss of generality, we can write the random matrix $\tilde{M}$ as the following form $$\tilde{M}=\left(
\begin{array}{ccccc}
	\tilde{M}'_{1} & 0 & \cdots &  0  & Y_1 \\
	0 & \tilde{M}'_{2} &\cdots &  0 &  Y_2 \\
	\vdots & \vdots& \ddots & \vdots & \vdots \\
	0 &  0 &\cdots &  \tilde{M}'_{m}  & Y_m\\
         Y_1^T &  Y_2^T &\cdots &  Y_m^T & \xi_{nn} \\
\end{array}
\right),$$ 
where $\tilde{M}'_{l}$ is the random matrix corresponding to $G'_l=(V'_l,E'_l)$. For $l=1,2,\ldots,m$, let $\mathbb{C}_l$ denote the event that there exists an eigenfunction of $\tilde{M}'_l$ has zeros and $\mathbb{D}$ denote the event that the spectrum of $\tilde{M}^{(nn)}$ is simple. We compute
\begin{equation*}
\begin{aligned}
          \text{Pr}[\mathbb{B}_n]&\leq \sum_{l=1}^m\text{Pr}[\mathbb{B}_n\cap\mathbb{C}_l]+\text{Pr}[\mathbb{B}_n\cap\left(\cup_{l}^{m}\mathbb{C}_l \right)^c]=\text{Pr}[\mathbb{B}_n\cap\left(\cup_{l}^{m}\mathbb{C}_l \right)^c]\\
          &=\text{Pr}[\mathbb{B}_n\cap\left(\cup_{l}^{m}\mathbb{C}_l \right)^c\cap\mathbb{D}]+\text{Pr}[\mathbb{B}_n\cap\left(\cup_{l}^{m}\mathbb{C}_l \right)^c\cap\mathbb{D}^c]\\
          &=\text{Pr}[\mathbb{B}_n\cap\left(\cup_{l}^{m}\mathbb{C}_l \right)^c\cap\mathbb{D}],    
\end{aligned}
\end{equation*}
where the above equalities are both by induction.

Assume the neighborhood of $n$ in $G$ are $\{i_1,i_2,\ldots,i_p\}$. Conditioning on all random variables except $\{\xi_{i_1,n},\xi_{i_2,n},\ldots\xi_{i_p,n}\}$ such that the event $\left(\cup_{l}^m\mathbb{C}_l\right)^c\cap\mathbb{D}$ happens, i.e., for any $l$, all eigenfunctions of $\tilde{M}'_l$ have no zeros and the spectrum of $\tilde{M}^{(nn)}$ is simple. Let $g$ be an eigenfunctions of $\Tilde{M}$ corresponding to $\lambda$. If $g$ is zero on vertex $n$, then the vector $\tilde{g}:=\left(g(1),g(2),\ldots,g(n-1)\right)\in \mathbb{R}^{n-1}$ is an eigenfunction of $\tilde{M}^{(nn)}$. By definition, we have $\sum_{j=1}^p\xi_{i_j,n}g(i_j)=0$.
Let $\{h_{l}\}_{l=1}^{n-1}$ be $n-1$ orthonormal eigenfunction of $\tilde{M}^{(nn)}$. By the above fact,  we compute 
 \begin{equation*}
         \text{Pr}\left[\mathbb{B}_n\right]\leq \sum_{l=1}^{n-1}\text{Pr}\left[\sum_{j=1}^p\xi_{i_j,n}h_l(i_j)=0\right].
 \end{equation*}
Since  for any $1\leq l\leq m$, all eigenfunctions of $\tilde{M}'$
 have no zeros. We assume $h_1(i_1)\neq 0$ without loss of generality. Then we have
  \begin{equation*}
    \text{Pr}\left[\sum_{j=1}^p\xi_{i_j,n}h_1(i_j)=0\right]=\text{Pr}\left[\xi_{i_1,n}=\frac{-\sum_{j=2}^p\xi_{i_j,n}h_1(i_j)=0}{h_{1}(i_1)}\right]=0.
 \end{equation*}
For any $2\leq l\leq n-1$, it is the same to prove that   $\text{Pr}\left[\sum_{j=1}^p\xi_{i_j,n}h_l(i_j)=0\right].$ This shows that $$\text{Pr}[\mathbb{B}_n]=0.$$
This concludes the proof.
\end{proof}

The following lemma shows the connection between strong nodal domains of eigenfunctions and the multi-way Cheeger constants. 
\begin{lemma}\label{lemma:main}
    Let $G=(V,E,w,\mu,\kappa)$ be a weighted graph with $\kappa\geq 0$. Let $f$ be an eigenfunction corresponding to the eigenvalue $\lambda_k(G)$. Then we have $$\rho_{m}(G)\leq \sqrt{2\tau_G\lambda_k(G)},$$where $m:=\mathfrak{S}(f)$ is the number of the strong nodal domains of $f$.
\end{lemma}
This lemma is an extension of Theorem 5 in \cite{DJM12} since we allow that the potential function is non-negative and $\tau_G$ does not have to be $1$. The proofs of them are almost the same. So we omit it here.

The above lemma shows that if we want to give a good upper of multi-way Cheeger constants, we can do this by finding an eigenfunction with many strong nodal domains. The following theorem gives a lower bound of the number of the strong nodal domains of eigenfunctions \cite{Berkolaiko,XY12}. Exactly, they proved a better low bound. For the sake of application, we only give the following form.
\begin{theorem}\label{thm:lower-bound}
    Let $G=(V,E,w,\mu,\kappa)$ be a connected weighted graph. The spectrum of its normalized Laplacian $\mathfrak{L}$ is listed as below:
    $$\lambda_1\leq \lambda_2 \leq \cdots\lambda_{k-1}<\lambda_{k}=\lambda_{k+1}=\cdots=\lambda_{k+r-1}< \lambda_{k+r}\leq \cdots \leq \lambda_{n}.$$
    Let $f_k$ be an eigenfunction corresponding to $\lambda_k$ with no zeros. Then the number of the strong nodal domains of $f_k$ is at least $k+r-1-\ell$ where $\ell=|E|-|V|+1$ is the cyclomatic number of $G$.
\end{theorem} 
Now, we prove the main result of the section.
\begin{proof}[Proof of Theorem \ref{thm:main}]
   We take a sequence of edge weights $\{w^{(m)}\}_{m=1}^{\infty}$ and non-negative potential functions $\{\kappa^{(m)}\}_{m=1}^{\infty}$ such that $w^{(m)}\rightarrow w$ and $\kappa^{(m)}\rightarrow \kappa$ as $m$ tends to infinity. For any $m$, we  define a weighted graph $G^{(m)}:=(V,E,w^{(m)},\mu,\kappa^{(m)})$. By Lemma \ref{lemma:nozero}, we can assume that all eigenfunctions of the normalized Laplacian $\mathfrak{L}^{(m)}$ of the graph $G^{(m)}$ have no zeros.  Denote the all eigenvalues of $\mathfrak{L}^{(m)}$ by $\{\lambda_{i}^{(m)}\}_{i=1}^n$.
    
   Fixed any $k$, let $f^{(m)}_k$ be an eigenfunction of $\mathfrak{L}^{(m)}$ corresponding to $\lambda^{(m)}_k$. By Theorem \ref{thm:lower-bound}, we have the number of the strong nodal domains of $f^{(m)}_k$ is at least $k-\ell$, where $\ell=|E|-|V|+1$.  
    By Lemma \ref{lemma:main}, we have 
    $$\rho_{k-\ell}(G^{(m)})\leq \rho_{\mathfrak{S}(f^{(m)}_k)}(G^{(m)})\leq \sqrt{2\tau_{G^{(m)}}\lambda_k^{(m)}},$$
    where the above first inequality is by Proposition \ref{pro:mono}. By definition of the multi-way Cheeger constant and the continuity of eigenvalues, we have $$\rho_{k-\ell}\left(G^{(m)}\right)\rightarrow \rho_{k-\ell}(G),\quad \tau_{G^{(m)}}\rightarrow \tau_G\quad \text{\quad and \quad } \lambda^{(m)}_{k}\rightarrow \lambda_k(G)$$ as $m$ tends to infinity. 
    This proves the inequality $$\rho_{k-\ell}(G)\leq \sqrt{2\tau_G\lambda_{k}(G)}.$$

\end{proof}
At the end of this section, we give a sequence of weighted graphs to show the advantage of our Theorem \ref{cor:main}. In the following examples, we assume all the potential functions are zero and $\mu_i=\sum_{j\sim i}w_{ij}$. So we omit them in the definition of the sequence of weighted graphs.  Moreover, the condition on vertex measure $\mu$ implies every set $A$ of any weighted graph must have $\Phi_G(A)\leq 1$ by definition. So there is a very trivial bound that $\rho_k(G)\leq 1$ for any $k$.
\begin{example}\label{ex:1}
    For $n\geq 3$, let $G_n=(V_n,E_n,w_n)$ be a sequence of graphs with 
    \begin{equation*}
\left\{
    \begin{aligned}
        &V_n=\{1,2,\ldots,2n+2\},\\
         & E_n=\big\{\{i,i+1\}:\text{ $i=1,2,\ldots,2n$}
  \big\}\cup\big\{\{1,2n+1\},\{n+1,2n+2\}\big\},\\
        &  w_n(i,i+1)=a^{i-1}, \quad\quad  \text{for $1\leq i \leq n$},\\
         &  w_n(i,i+1)=a^{2n-i}, \quad\quad  \text{for $n+1\leq i \leq 2n$},\\
        &  w_n(1,2n+1)=1, \quad\quad w_n(n+1,2n+2)=1.
    \end{aligned}
    \right.
\end{equation*}
where $a=1.1$.

We draw $G_3$ in Figure \ref{fig:3}. The first one is to show the vertices label and the second one is to show the edge weight.
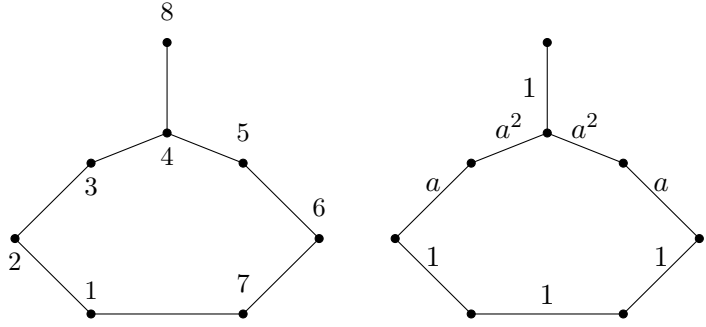
\begin{figure}[!htp]
	\centering
	\tikzset{vertex/.style={circle, draw, fill=black!20, inner sep=0pt, minimum width=3pt}}
	\begin{tikzpicture}[scale=1.0]
 	 \draw (-1,0) -- (-2,1) node[midway, above, black]{$ $}
		-- (-1,2) node[midway, above, black]{$ $}	-- (0,2.4) node[midway, above, black]{$ $}
  	-- (1,2) node[midway, above, black]{$ $}
   	-- (2,1) node[midway, above, black]{$ $}
    	-- (1,0) node[midway, above, black]{$ $}
     	-- (-1,0) node[midway, above, black]{$ $};
		\draw (0,2.4) -- (0,3.6) node[midway,left, black]{$ $} ;

		\node at (-1,0) [vertex, label={[label distance=0mm]90: \small $1$}, fill=black] {};
           \node at (-2,1) [vertex, label={[label distance=0mm]270: \small $2$} ,fill=black] {};
		\node at (-1,2) [vertex, label={[label distance=0mm]270: \small $3$} ,fill=black] {};
		\node at (0,2.4) [vertex, label={[label distance=0mm]270: \small $4$} ,fill=black] {};
	\node at (1,2) [vertex, label={[label distance=1mm]90: \small $5$} ,fill=black] {};
 \node at (2,1) [vertex, label={[label distance=1mm]90: \small $6$} ,fill=black] {};
  \node at (1,0) [vertex, label={[label distance=1mm]90: \small $7$} ,fill=black] {};
   \node at (0,3.6) [vertex, label={[label distance=1mm]90: \small $8$} ,fill=black] {};

	 \draw (4,0) -- (3,1) node[midway, above, black]{$1$}
		-- (4,2) node[midway, above, black]{$a$}	-- (5,2.4) node[midway, above, black]{$a^2$}
  	-- (6,2) node[midway, above, black]{$a^2$}
   	-- (7,1) node[midway, above, black]{$a$}
    	-- (6,0) node[midway, above, black]{$1$}
     	-- (4,0) node[midway, above, black]{$1$};
		\draw (5,2.4) -- (5,3.6) node[midway,left, black]{$1$} ;

		\node at (4,0) [vertex, label={[label distance=0mm]90: \small $ $}, fill=black] {};
           \node at (3,1) [vertex, label={[label distance=0mm]270: \small $ $} ,fill=black] {};
		\node at (4,2) [vertex, label={[label distance=0mm]270: \small $ $} ,fill=black] {};
		\node at (5,2.4) [vertex, label={[label distance=0mm]270: \small $ $} ,fill=black] {};
	\node at (6,2) [vertex, label={[label distance=1mm]90: \small $ $} ,fill=black] {};
 \node at (7,1) [vertex, label={[label distance=1mm]90: \small $ $} ,fill=black] {};
  \node at (6,0) [vertex, label={[label distance=1mm]90: \small $ $} ,fill=black] {};
   \node at (5,3.6) [vertex, label={[label distance=1mm]90: \small $ $} ,fill=black] {};

	\end{tikzpicture}
	\caption{$G_3=(V_3,E_3,w_3,\mu_3)$.}
	
	\label{fig:3}
\end{figure}

After some routine calculations, we get $\rho_2(G_n)\geq \theta$ for any $n\geq 3$, where $\theta$ is a constant which is smaller than $0.1$. By Cheeger inequality, we have$$\lambda_k(G_n)\geq \lambda_2(G_n)\geq \frac{\theta^2}{2}.$$ When $k$ is big enough, the bounds $\sqrt{C_0\lambda_k(G_n)\log k}$ and $Ck^2\sqrt{\lambda_k(G_n)}$ are both bigger than $1$. Then the previous results will give upper bounds weaker than the trivial bound. But the bound $\sqrt{2\tau_G\lambda_k(G_n)}$ in Theorem \ref{thm:main} gives a non-trivial bound if $\lambda_k(G_n)\leq\frac{1}{2}$, where the existence of this eigenvalue for any $k$ is by taking $n$ is big enough.
\end{example}

\section{Higher-order Cheeger inequalities for special graphs}
\label{sec:main3}
We only consider the vertex measure $\mu \equiv 1$ in this section. Hence, for any weighted graph, we write it as $G=(V,E,w,\kappa)$ for simplicity. In this situation, the normalized Laplacian is equal to the Laplacian.

The product of two weighted graphs  $G_1=(V^{(1)},E^{(1)},w^{(1)},\kappa^{(1)}) \text{ and } G_2=(V^{(2)},E^{(2)},w^{(2)},\kappa^{(2)})$ is the weighted graph $G=(V,E,w,\kappa)$ which is defined as follows:
\begin{equation*}
\left\{
    \begin{aligned}
        &V:=V^{(1)}\times V^{(2)},\\
         & (x,y)\sim (x',y') \text{ if and only if }\text{$x=x'$ and $y\sim y'$ or $x\sim x'$ and $y= y'$},\\
        &  w_{(x,y),(x,y')}:=w^{(2)}_{yy'}, \quad\quad  w_{(x,y),(x',y)}:=w^{(1)}_{xx'},\\
        &  \kappa_{(x,y)}:=\kappa^{(1)}_x+\kappa^{(2)}_y.
    \end{aligned}
    \right.
\end{equation*}

\begin{theorem}\label{thm:main3}

    Let $G_1=(V^{(1)},E^{(1)},w^{(1)},\kappa^{(1)}) $  be a tree and $G_2=(V^{(2)},E^{(2)},w^{(2)},\kappa^{(2)})$ be a bipartite graph with $\kappa^{(1)}\geq 0$ and $\kappa^{(2)}\geq 0$. Denote $|V_1|=n_1$ and $|V_2|=n_2$. Define $G=(V,E,w,\kappa)$ be the product graph of $G_1$ and $G_2$.  Assume $\{\lambda_{i}\}_{i=1}^{n_1n_2}$, $\{\lambda^{(1)}_{i}\}_{i=1}^{n_1}$ and $\{\lambda^{(2)}_{i}\}_{i=1}^{n_2}$ are eigenvalues of Laplacian of $G$,  $G_1$ and $G_2$ in ascending order, respectively. If $\lambda^{(2)}_{n_2}<\lambda^{(1)}_{k+1}-\lambda^{(1)}_{k}$, we have $$\rho_{kn_2}(G)\leq \sqrt{2\tau_{G}\lambda_{kn_2}}.$$
\end{theorem}

Before proving this theorem, we give some analysis of the strong nodal domains.
\begin{lemma}\label{lemma:nodal count}
    Given two functions $f:V^{(1)}\to \mathbb{R}$ and $g:V^{(2)}\to \mathbb{R}$, we define $h:V\to \mathbb{R}$ as follows $$h\left( (x,y)\right)=f(x)g(y),\quad\quad \text{for any }(x,y)\in V:=V^{(1)}\times V^{(2)}. $$ Then we have $$\mathfrak{S}(h)=\mathfrak{S}(f)\mathfrak{S}(g).$$
\end{lemma}
\begin{proof}
   Assume $\mathfrak{S}(f)=k_1$ and $\mathfrak{S}(g)=k_2$. Let $\{S_i\}_{i=1}^{k_1}$ and $\{W_j\}_{j=1}^{k_2}$ be the strong nodal domains of $f$ and $g$, respectively. We prove that the strong nodal domains of  $h$ are exactly $$\{S_i\times W_j\}_{i=1,\ldots,k_1,j=1,\ldots,k_2}.$$
   
  Fixed $i$ and $j$, without loss of generality, we assume $f$ is positive on $S_i$ and $g$ is positive on $W_j$. Then we have  $h$ is positive on $S_i\times W_j$. For any $(x,y)$ and $(x',y')$ in $S_i\times W_j$, by definition of the strong nodal domains, there exist two sequences of vertices $\{x_{l}\}_{l=1}^{m_1}\subset S_i$ and $\{y_{l}\}_{l=1}^{m_2}\subset W_j$ such that  $x\sim x_1 \sim x_2 \sim \cdots \sim x_{m_1}\sim x'$ and $y\sim y_1 \sim y_2 \sim \cdots \sim y_{m_2}\sim y'$. Then we have 
  $$(x,y)\sim (x_1,y) \sim (x_2,y) \sim \cdots \sim (x_{m_1},y)\sim (x',y)\sim (x',y_1) \sim (x',y_2) \sim \cdots \sim (x',y_{m_2})\sim (x',y').$$ This shows that $S_i\times W_j$ belongs to one strong nodal domain of $h$.

  If there exist $(x,y)\in S_{i_1}\times W_{j_1} $ and $(x',y')\in S_{i_2}\times W_{j_2} $ such that $(x,y)\sim (x',y')$. Without loss of generality, we assume $x=x'$, $i_1=i_2$ and $y\sim y'$. If $h\left((x,y)\right)h\left((x',y')\right)>0$, then $g(y)g(y')>0$. By the fact that $y\sim y'$, we have $W_{j_1}=W_{j_2}$. This implies that $ S_{i_1}\times W_{j_1} $ and $ S_{i_2}\times W_{j_2} $ belongs to different strong nodal domains if $i_1\neq i_2$ or  $j_1\neq j_2$.
  
  Since the set consists of nonzero vertices of $h$ is $\cup_{i=1}^{k_1}\cup_{j=1}^{k_2}S_i\times W_j$, then the all strong nodal domains of $h$ are  $\{S_i\times W_j\}_{i=1,\ldots,k_1,j=1,\ldots,k_2}.$
\end{proof}

\begin{proof}[Proof of Theorem \ref{thm:main3}]
    By Lemma \ref{lemma:nozero}, there exist a sequence on edge weights $\{w^{(1,l)}\}_{l=1}^{\infty}$ of $E^{(1)}$ and a sequence potential functions $\{\kappa^{(1,l)}\}_{l=1}^{\infty}$ on $V^{(1)}$ such that $$\lim_{l\to \infty}w^{(1,l)}=w^{(1)}\text{ and } \lim_{l\to \infty}\kappa^{(1,l)}=\kappa^{(1)},$$
    with all eigenfunctions of the Laplacian $L^{(1,l)}$ of $G^{(1,l)}:=(V,E,w^{(1,l)},\kappa^{(1,l)})$ have no zeros for any $l$. Let $\{\lambda_i^{(1,l)}\}_{i=1}^{n_1}$ be the all eigenvalues of $L^{(1,l)}$ in ascending order. Moreover, we can assume that $\lambda^{(2)}_{n_2}<\lambda^{(1,l)}_{k+1}-\lambda^{(1,l)}_{k}$ by the continuity of eigenvalues. Let $f^{(l)}:V^{(1)}\to \mathbb{R}$  $\left(\text{resp., }g:V^{(2)}\to \mathbb{R}\right)$ be an eigenfunction of Laplacian of $G^{(1,l)}$ $\left(\text{resp., }G^{(2)}\right)$ corresponding to $\lambda_k^{(1,l)}$ $\left(\text{resp., } \lambda_{n_2}^{(2)}\right)$. By  Theorem \ref{thm:lower-bound}, we have $\mathfrak{S}(f^{(l)})\geq k$. By Theorem 4.1 in \cite{BLS05}, we have $\mathfrak{S}(g)=n_2$.
    Define $h^{(l)}:V\to \mathbb{R}$ as follows $$h^{(l)}\left( (x,y)\right)=f^{(l)}(x)g(y),\quad\quad \text{for any }(x,y)\in V=V^{(1)}\times V^{(2)}. $$ Let $G^{(0,l)}=(V,E,w^{(0,l)},\kappa^{(0,l)})$ be the product of $G^{(1,l)}$ and $G^{(2)}$.  By direct compute, we have $h^{(l)}$ is an eigenfunction of the Laplacian of $G^{(0,l)}$ corresponding to $\lambda^{(1,l)}_k+\lambda^{(2)}_{n_2}$. Because $\lambda^{(2)}_{n_2}<\lambda^{(1,l)}_{k+1}-\lambda^{(1,l)}_{k}$, we have $\lambda_{kn_2}^{(0,l)}=\lambda^{(1,l)}_k+\lambda^{(2)}_{n_2}$, where $\{\lambda^{(0,l)}_{i}\}_{i=1}^{n_1n_2}$ are the all eigenvalues of the Laplacian of $G^{(0,l)}$ in ascending order. We use Lemma \ref{lemma:main} to get  $$\rho_{kn_2}\left(G^{(0,l)}\right)\leq  \rho_{\mathfrak{S}\left(f^{(l)}\right)\mathfrak{S}(g)}\left(G^{(0,l)}\right)=\rho_{\mathfrak{S}\left(h^{(l)}\right)}\left(G^{(0,l)}\right)\leq \sqrt{2\tau_{G^{(0,l)}}\lambda_{kn_2}^{(0,l)}},$$ 
where the above first inequality is by Proposition \ref{pro:mono} and the above first equality is by Lemma \ref{lemma:nodal count}. We compute $$\lim_{l\to \infty}\lambda_{kn_2}^{(0,l)}=\lambda_{kn_2},\quad\quad\lim_{l\to \infty}\tau_{G^{(0,l)}}=\tau_{G}\quad\text{ and }\quad\lim_{l\to \infty}\rho_{kn_2}\left(G^{(0,l)}\right)=\rho_{kn_2}\left(G\right).$$
This concludes the proof.
\end{proof}

\section{Lower bounds for multi-way Cheeger constants }
\label{sec:main4}

In this section, we assume $\kappa\equiv0$ and $|V|\geq 3$. We denote the weighted graph $G=(V,E,w,\mu,\kappa)$ by $G=(V,E,w,\mu)$ for short. 

Define the normalized adjacency matrix of $G=(V,E,w,\mu)$ by $\mathcal{A}:=M^{-1}A$, where $A$ is the adjacency matrix of graph $G$ and $M$ is the diagonal matrix with $M_{ii}=\mu_i$. We assume that all eigenvalues of $\mathcal{A}$ are listed below 
$$\eta_n\leq\eta_{n-1}\leq \cdots\leq \eta_{2}\leq \eta_1.$$ Let $\eta=\max\{|\eta_2|,|\eta_{n}|\}$.

\begin{theorem}\label{thm:main4}
       Let $G=(V,E,w,\mu)$ be a weighted graph.  Then we have $$\rho_k(G)\geq (\tau_{min}-\eta)(1-\frac{1}{k}),$$
       where $\tau_{min}=\min_{i\in V}\frac{d(i)}{\mu_i }$.
\end{theorem}
When $\mu_i=d(i)$, we have $1-\eta_i=\lambda_i(G)$ and $\tau_{min}=1$. By the fact that $\eta_2$ must be non-negative if the graph is not complete, we have the following corollary.

\begin{cor}\label{cor:main4}
       Let $G=(V,E,w,\mu)$ be a weighted and not complete graph with $\mu_i=d(i)$.  Then we have $$\rho_k(G)\geq (1-\frac{1}{k})\min \{\lambda_2(G), 2-\lambda_n(G)\}.$$
\end{cor}

The following proof is inspired by the proof of Lemma 2.3 in \cite{AC86}.
\begin{proof}[Proof of Theorem \ref{thm:main4}]
Given any $(A_1,A_2,\ldots,A_k)$ such that $\emptyset\neq A_1,\ldots,A_k\subset V\text{ and }A_i\cap A_j= \emptyset$. For any set $S\subset V$, we define $\mu(S):=\sum_{x\in S}\mu_x$. Without loss of generality, we assume $\mu(A_1)\leq \mu(A_2)\leq \cdots \leq \mu(A_{k-1})\leq \mu(A_k)$. Then we directly have $\mu(A_1)\leq \frac{\mu(V)}{k}$. Define $\alpha:=\frac{\mu(A_1)}{\mu(V)}$. By definition, we have $\alpha\leq \frac{1}{k}$.
 Define a function $f:V\to \mathbb{R}$ as follows \begin{equation*}
f(x)=\left\{
    \begin{aligned}
        &\frac{1}{\mu(A_1)}\quad\quad\quad && x\in A_1,\\
        &  \frac{-1}{\mu(V\setminus A_1)}\quad\quad\quad && x\in V\setminus A_1.
    \end{aligned}
    \right.
\end{equation*}
 Since $\sum_{x\in V}\mu_xf(x)=0$, we have $$|\langle Af,f \rangle|\leq \eta \sum_{x\in V}\mu_xf^2(x)=\eta \left(\frac{1}{\mu(A_1)}+\frac{1}{\mu(V\setminus A_1)}\right).$$ 
 Next, we compute
 \begin{equation*}
     \begin{aligned}
         |\langle Af,f \rangle|&=\left|\sum_{x\in V}d_xf^2(x)-\sum_{\{x,y\}\in E}w_{xy}\left(f(x)-f(y)\right)^2\right|\\
         &=\left|\sum_{x\in V}d_xf^2(x)-E(A_1,\overline{A_1})\left(\frac{1}{\mu(A_1)}+\frac{1}{\mu(V\setminus A_1)}\right)^2\right|,
     \end{aligned}
 \end{equation*}
where $E(A_1,\overline{A_1})=\sum_{x\in A_1}\sum_{y\in \overline{A_1}}w_{xy}$. Combining the above two inequalities, we obtain
 \begin{equation*}
     \begin{aligned}
        E(A_1,\overline{A_1})\left(\frac{1}{\mu(A_1)}+\frac{1}{\mu(V\setminus A_1)}\right)^2&\geq \sum_{x\in V}d_xf^2(x)-\eta \left(\frac{1}{\mu(A_1)}+\frac{1}{\mu(V\setminus A_1)}\right)\\
         &\geq (\tau_{min}-\eta) \left(\frac{1}{\mu(A_1)}+\frac{1}{\mu(V\setminus A_1)}\right).
     \end{aligned}
 \end{equation*}
 This implies 
$$ E(A_1,\overline{A_1})\geq (\tau_{min}-\eta)\alpha(1-\alpha)\mu(V).$$
Then we have $$ \max_{1\leq i\leq k}\Phi_{G}(A_i)\geq \Phi_G(A_1)=\frac{E(A_1,\overline{A_1})}{\mu(A_1)}\geq (\tau_{min}-\eta)(1-\alpha)\geq (\tau_{min}-\eta)(1-\frac{1}{k}).$$
This theorem is followed by the arbitrariness of $(A_1,A_2,\cdots A_k)$.
\end{proof}

\section{Higher-order Cheeger inequalities for signed graph}
\label{sec:signed}
In this section, we show that it is direct to extend Theorem \ref{thm:main} to the case of the signed graphs. 

A signed graph $\Gamma=(G,\sigma)$ is a weighted graph $G=(V,E,w,\mu,\kappa )$ with a signature $\sigma:E\to \{-1,+1\}$.  The adjacency matrix of $\Gamma$ is $A^{\sigma}=\{a_{ij}\}_{1\leq i,j\leq n}$ where $a_{ij}=\sigma_{ij}w_{ij}$ if $i\sim j$ and $a_{ij}=0$ if $i\nsim j$. The weighted degree of a vertex $i$ is defined by $d(i)=\sum_{j\sim i}w_{ij}$.  Let $D$, $K$ and $M$ be defined the same as in Section \ref{sec:pre}. The normalized Laplacian of $\Gamma$ is defined by  $\mathfrak{L}^{\sigma}=M^{-1}(D-A^{\sigma}+K)$.

Next, we review the definition of signed multi-way Cheeger constants. Given any two subsets $V_1,V_2\subset V$, we define the following notations
$$  |E^{+}(V_1,V_2)|:=\sum_{x\in V_1}\sum_{\substack{y\in V_2\\ \sigma_{xy}=+ 1}}w_{xy},\quad|E^{-}(V_1,V_2)|:=\sum_{x\in V_1}\sum_{\substack{y\in V_2\\ \sigma_{xy}=-1}}w_{xy}\quad\text{ and }\quad|\partial V_1|:=\sum_{x\in V_1}\sum_{y\not\in V_1}w_{xy},$$
and write $|E^{\pm }(V_1)|=|E^{\pm }(V_1,V_1)|$  for simplicity. Let $\beta^\sigma(V_{1},V_{2})$ be defined as follows: \[\beta^\sigma(V_{1},V_{2}):=\frac{2|E^+(V_{1},V_{2})|+|E^-(V_{1})|+|E^-(V_{2})|+|\partial(V_{1}\cup V_{2})|}{\sum_{x\in V_1\cup V_2}\mu_x}.\]

Given any $1\leq k\leq n$, the $k$-way signed Cheeger constant $\rho_k^\sigma(\Gamma)$ is defined as follows \cite{AL20}
\begin{equation*}
    \rho_k^\sigma(\Gamma):=\min_{\{(V_{2i-1}, V_{2i})\}_{i=1}^k}\max_{i=1,\ldots,k}\beta^\sigma(V_{2i-1},V_{2i}),
\end{equation*} 
where the minimum is taken over all possible $k$-sub-bipartitions, i.e., $(V_{2i-1}\cup V_{2i})\bigcap (V_{2j-1}\cup V_{2j})=\emptyset$ for any $i\neq j$, and $V_{2l-1}\cup V_{2l}\neq \emptyset$, $V_{2l-1}\cap V_{2l}=\emptyset$ for any $l$.

The proof of Theorem \ref{thm:main} uses the concept of strong nodal domains. The definition and the lower bound of strong nodal domains of signed graphs were established in \cite{GL23} and the signed version of Lemma \ref{lemma:main} was established in \cite[Theorem 8]{GLZ23}. After the preparation, we use Lemma \ref{lemma:nozero} the same as the proof of Theorem \ref{thm:main} to get the following theorem.
\begin{theorem}
    Let $\Gamma=(G,\sigma)$ be a connected signed graph with $G=(V,E,w,\mu,\kappa)$ and $\kappa\geq 0$. Let $\{\lambda^{\sigma}_i\}_{i=1}^n$ be the all eigenvalues of the normalized Laplacian of $\Gamma$ in ascending order. Then we have
     $$\rho_{k-\ell}^{\sigma}(\Gamma)\leq \sqrt{2\tau_{\Gamma}\lambda^{\sigma}_k},$$ 
 where $\ell=|E|-|V|+1$ and $\tau_\Gamma=\max_{i\in V}\frac{d(i)}{\mu_i}$.
\end{theorem}

\section*{Acknowledgement}
This work is supported by the National Key R and D Program of China 2020YFA0713100, the National Natural Science Foundation of China (No. 12031017 and No. 12431004), and Innovation Program for Quantum Science and Technology 2021ZD0302902. We are very grateful to Shiping Liu and Guangyi Zou for their inspiring discussions, especially on the proof of Lemma \ref{lemma:nozero}.

\end{document}